\documentclass[notitlepage,11pt,reqno]{amsart}
\usepackage{amssymb,amsmath,amscd,xspace}
\usepackage[mathscr]{eucal}
\usepackage[breaklinks=true]{hyperref}%To get clickable links to papers

\newcommand{\F}{\ensuremath{\mathbb{F}}\xspace}
\newcommand{\N}{\ensuremath{\mathbb{N}}\xspace}
\newcommand{\Z}{\ensuremath{\mathbb{Z}}\xspace}

\newcommand{\proj}{\ensuremath{\mathbb{P}}\xspace}

\theoremstyle{plain}
\newtheorem{theorem}[equation]{Theorem}
\newtheorem{lemma}[equation]{Lemma}
\newtheorem{prop}[equation]{Proposition}

\theoremstyle{definition}
\newtheorem{definition}[equation]{Definition}
\newtheorem{remark}[equation]{Remark}

\numberwithin{equation}{section}

\begin{document}
\title{Vector spaces as unions of proper subspaces}
\author{Apoorva Khare}
\email{apoorva@math.ucr.edu}
\address{Dept.~of Mathematics, Univ.~of California at Riverside,
Riverside, CA 92521}
\date{\today}

\begin{abstract}
In this note, we find a sharp bound for the minimal number (or in
general, indexing set) of subspaces of a fixed (finite) codimension
needed to cover any vector space $V$ over any field. If $V$ is a finite
set, this is related to the problem of partitioning $V$ into subspaces.
\end{abstract}
\maketitle
%}}}

%{{{1 Section 1 - The main theorem
\section{The main theorem}

Consider the following well-known problem in linear algebra (which is
used, for example, to produce vectors not on root hyperplanes in Lie
theory):\medskip

{\it No vector space over an infinite field is a finite union of proper
subspaces.}\medskip

\noindent The question that we answer in this short note, is:\medskip

{\it Given any vector space $V$ over a ground field $\F$, and $k \in \N$,
what is the smallest number (or in general, indexing set) of proper
subspaces of codimension $k$, whose union is $V$?}\medskip

To state our main result, we need some definitions.

\begin{definition}\hfill
\begin{enumerate}
\item Compare two sets $I,J$ as follows: $J > I$ if there is no
one-to-one map $f : J \to I$. Otherwise $J \leq I$.

\item Given a vector space $V$ over a field $\F$, define $\proj(V)$ to be
the set of lines in $V$; thus, $\proj(V)$ is in bijection with $(V
\setminus \{ 0 \}) / \F^\times$.
\end{enumerate}
\end{definition}

This paper is devoted to proving the following theorem.

\begin{theorem}\label{Tvecunion}
Suppose $V$ is a vector space over a field $\F$, and $I$ is an indexing
set. Also fix $1 \leq k < \dim_\F V$, $k \in \N$.
Then $V$ is a union of ``$I$-many" proper subspaces of codimension at
least $k$, if and only if $I \geq \nu(\F, V, k)$, where
\[ \nu(\F,V,k) :=
  \begin{cases}
    \lceil |\proj(V)| / |\proj(V / \F^k)| \rceil, &\text{if $|V| <
    			\infty$;}\\
    \N,                       &\text{if $|\F| = \dim_\F V = \infty$;}\\
    \F^k \coprod \{ \infty \},                    &\text{otherwise.}\\
  \end{cases} \]
\end{theorem}

\noindent (We will see in the proof, that the countable cover in the case
of $\dim_\F V = |\F| = \infty$ is different in spirit from the
constructions in the other cases.)
We now mention a few examples and variants, before proving the theorem in
general.
\begin{enumerate}
\item No finite-dimensional vector space over $\mathbb{R}$ (hence, also
over $\mathbb{C}$) is a union of countably many proper subspaces. Here is
a simple measure-theoretic proof by S.~Chebolu: suppose $V = \bigcup_{n >
0} V_n$, with $V_n \subsetneq V\ \forall n \in \N$. Let $\mu$ be the
Lebesgue measure on $V$; recall that $\mu$ is countably subadditive. We
now get a contradiction:
\[ \mu(V) = \mu \left( \bigcup_n V_n \right) \leq \sum_{n \in \N}
\mu(V_n) = 0, \]

\noindent since each $V_n$ has measure zero, being a proper
subspace.\medskip

\item On the other hand, suppose $V$ is a finite-dimensional vector space
over a finite field $\F = \F_q$ (with $q$ elements); how many proper
subspaces would cover it? (Hence, $k=1$ here.) The answer is the same for
all $V$; we mention a proof (by R.~Walia; see \cite{Luh}) for the
simplest example of $V_2 = \F_q^2$.\medskip

\begin{lemma}\label{Llines}
$V_2$ is a union of $q+1$ lines (but not $q$ lines).
\end{lemma}

\begin{proof}
Consider the lines spanned by $(1, \alpha)$ (for each $\alpha \in \F_q$)
and $(0,1)$. These are $q+1$ lines, and each pair of lines has only the
origin in common (since two points determine a line). Since each line has
$q$ points, the union of all these lines has size $1 + (q+1)(q-1) = q^2$
(where the ``1"counts the origin). This counting argument also shows that
a smaller number of lines can not cover all of $V_2$.
\end{proof}

\begin{remark}
Thus, we should really think of $q+1$ as $\F \coprod \{ \infty \} =
\proj(\F^2)$.
\end{remark}

\item Note that the proof of Lemma \ref{Llines} shows that the $q+1$
lines actually provide a {\it partition} of the finite vector space $V_2$
- namely, a set of subspaces that are pairwise disjoint except for the
origin, and cover all of $V$.

The theory of partitions of finite vector spaces has been extensively
studied - see, for instance, \cite{Beu, Bu, ZSSSE, Hed1, Hed2}. We remark
that this theory of partitions keeps track of the dimensions of the
subspaces involved. Moreover, it has applications in error-correcting
codes and combinatorial designs - see \cite[\S 1]{ZSSSE} for more
references.

\item There is a school of thought that considers vector spaces over
``$\F_1$ (the field with one element)", to morally be defined - and more
precisely, they are finite sets. The way to get results using this
philosophy, is to work the analogous results out for finite fields
$\F_q$, and take $q \to 1^+$ (though it is a non-rigorous procedure,
given that there usually is more than one generalization to $\F_q$).

As for our two problems, the results are clear: a set of size $>1$
(which is analogous to $\dim_{\F_q}(V) > 1$) is a union of two proper
subsets - where $2 = 1+1 = q+1$ - but not of one proper subset. The
analogue for codimension $k$ subspaces, is: how many subsets $W \subset
V$ with $|V \setminus W| \geq k$, does it take to cover $V$?

The answer to this question is 2 if $V$ is infinite, and if $|V| = n$,
then the answer is $\displaystyle \Big\lceil \frac{n}{n-k} \Big\rceil$.
Note that this is exactly the statement of Theorem \ref{Tvecunion} for
finite vector spaces $V$ in both cases, because $\proj(V) / \proj(V /
\F_q^k) = (q^n - 1) / (q^{n-k}-1)$, and for $0<k<n$,
\[ \lim_{q \to 1^+} \Big\lceil \frac{q^n - 1}{q^{n-k} - 1} \Big\rceil =
\Big\lceil \frac{n}{n-k} \Big\rceil. \]

\item The next variant involves modules over a finite-dimensional
$\F$-algebra $A$, when $\F$ is infinite. It generalizes Theorem
\ref{Tvecunion} when $k=1$.

\begin{prop}\label{Palg}
Suppose $\dim_{\F} A < \infty = |\F|$. Now if $M = \bigoplus_{i \in I} A
m_i$ is any direct sum of cyclic $A$-modules, then $M$ is a union of
``$J$-many" proper submodules if and only if it $M$ is not cyclic and $J
\geq \nu(\F,M,1)$.
\end{prop}

\noindent In other words, $M$ is a union of ``$\F$-many" proper
submodules if $I$ is finite (and not a singleton), and a countable union
if $I$ is infinite, since $\F$ and $\F \coprod \{ \infty \}$ are in
bijection if $\F$ is infinite.

\begin{proof}
If $M$ is cyclic, the result is clear, since some submodule must contain
the generator. So now assume that $M$ is not cyclic; note that each
cyclic $A$-module is a quotient of $A$, hence finite-dimensional. So if
$I$ is finite, then $\dim_\F M < \infty$, and every proper submodule is a
subspace of codimension between 1 and $\dim_\F M$. By Theorem
\ref{Tvecunion}, we then need at least ``$\F$-many" proper submodules to
cover $M$.

On the other hand, $|M| = |\F|$, and for each $m \in M$, we have the
submodule $Am$ containing it. The result follows (for finite $I$) if we
can show that $Am$ is a proper submodule for all $m \in M$. But if $Am =
M$ for some $m = \sum_{i \in I} a_i m_i$, then we can find $b_j \in A$
such that $b_j m = m_j\ \forall j \in I$, whence $b_j a_i = \delta_{ij}\
\forall i,j$.
This is a contradiction if $|I|>1$, since it implies that every $a_i \in
A^\times$ is a unit, hence annihilated only by $0 \in A$.

On the other hand, if $I$ is infinite, then for any sequence of subsets
\[ \emptyset = I_0 \subset I_1 \subset I_2 \subset \dots \subset I,
\qquad \bigcup_{n \in \N} I_n = I, \qquad I_n \neq I\ \forall n, \]

\noindent define the proper $A$-submodule $M_n := \bigoplus_{i \in I_n} A
m_i$. Then $M = \bigcup_{n \in \N} M_n$ yields a countable cover by
proper $A$-submodules. Evidently, finitely many proper submodules cannot
cover $M$ (again by Theorem \ref{Tvecunion}), since $\F$ is infinite, and
each submodule is a subspace as well.
\end{proof}

\item The last variant that we mention, is the following question, which
generalizes Theorem \ref{Tvecunion} for finite $\F$:\medskip

{\it Given a finitely generated abelian group $G$, how many proper
subgroups are needed to cover it?}\medskip

For instance, if $G = (\Z / 5 \Z) \oplus (\Z / 25 \Z) \oplus (\Z / 12
\Z)$, then $G$ has a quotient: $G \to (\Z / 5 \Z)^2 = \F_5^2 \to 0$. Now
to cover $G$ by proper subgroups, we can cover $\F_5^2$ by six lines (by
Theorem \ref{Tvecunion}), and lift them to a cover of $G$. Moreover, it
can be shown that $G$ cannot be covered by five or fewer proper
subgroups.

Thus, the question for abelian groups seems to be related to the question
for fields. We explore this connection in a later work \cite{Khchinese}.
\end{enumerate}
%}}}

%{{{1 Section 2 - Proof for infinite fields
\section{Proof for infinite fields}

In this section, we show Theorem \ref{Tvecunion} for infinite fields (and
some other cases too). Define {\it projective $k$-space} $\F P^k :=
\proj(\F^{k+1})$.

\begin{remark}\label{R2}\hfill
\begin{enumerate}
\item In what follows, we freely interchange the use of (cardinal)
numbers and sets while comparing them by inequalities. For instance, $I
\geq A/B$ and $I \geq n$ mean, respectively, that $I \times B \geq A$ and
$I \geq \{ 1, 2, \dots, n \}$. Similarly, $\dim_\F V$ may denote any
basis of $V$ - or merely its cardinality.

We also write $\cong$ below, for bijections between sets (in other
contexts and later sections, $\cong$ may also denote bijections of
$\F$-vector spaces).

\item $\F P^k$ is parametrized by the following lines:
\[ (1, \alpha_1, \dots, \alpha_k);\ (0, 1, \alpha_2, \dots, \alpha_k);\
\dots;\ (0,0, \dots, 0, 1), \]

\noindent where all $\alpha_i$ are in $\F$. If $\F$ is infinite, then
this is in bijection with each of the following sets: $\F, \F^k, \F
\coprod \{ \infty \}, \F^k \coprod \{ \infty \}$.
\end{enumerate}
\end{remark}\medskip

We now show a series of results, that prove the theorem when $\F$ is
infinite.

\begin{lemma}
($\F, V, k, I$ as above.) If $I \geq \F P^k$, then $V$ is a union of
``$I$-many" proper subspaces of codimension at least $k$, if and only if
$\dim_\F V > k$.
\end{lemma}

\begin{proof}
The result is trivial if $\dim_\F V \leq k$, and if not, then we start by
fixing any $\F$-basis $B$ of $V$. Fix $v_0, v_1, \dots, v_k \in B$, and
call the complement $B'$. Now define, for each $1 \leq i \leq k$ and each
$x = (0, \dots, 0, 1, \alpha_i, \alpha_{i+1}, \dots, \alpha_k) \in \F
P^k$, the codimension $k$-subspace $V_x$ of $V$, spanned by $B'$ and
$v_{i-1} + \sum_{j=i}^k \alpha_j v_j$.

We claim that $V = \bigcup_{x \in \F P^k} V_x$. Indeed, any $v \in V$ is
of the form $v' + \sum_{j=0}^k \beta_j v_j$, with $\beta_j \in \F\
\forall j$, and $v'$ in the span of $B'$. Now if $\beta_i$ is the first
nonzero coefficient, then $v \in V_x$, where $x = (0, \dots, 0, 1,
\beta_i^{-1} \beta_{i+1}, \dots, \beta_i^{-1} \beta_k)$, with the $1$ in
the $i$th coordinate.
\end{proof}

\begin{prop}\label{P1}
Suppose $I < \F^k \coprod \{ \infty \}$. If $I$ or $\dim_\F V$ is finite,
then $V$ cannot be written as a union of ``$I$-many" subspaces of
codimension $\geq k$.
\end{prop}

\begin{proof}
This proof is long - and hence divided into steps.
\begin{enumerate}
\item The first step is to show it for $k=1$. Suppose we are given $V$
and $\{ V_i : i \in I \}$. Suppose the result fails and we do have $V =
\bigcup_{i \in I} V_i$. We then seek a contradiction.
\begin{enumerate}
\item We first find a subcollection $\{ V_i : i \in I' \subset I \}$ of
subspaces that cover $V$, such that no $V_i$ is in the union of the
rest.\medskip

If $I$ is finite, this is easy: either the condition holds, or there is
some $V_i$ that is contained in the union of the others; now remove it
and proceed by induction on $|I|$.\medskip

The case of finite-dimensional $V$ is from \cite{Ani}. We need to use
induction on $d = \dim_\F V$ to prove the result. It clearly holds if $V
= \F^1$; now suppose that it holds for all $d < \dim_\F V$. We first
reduce our collection $\{ V_i : i \in I \}$ to a subcollection indexed by
$I' \subset I$, say, as follows:

Every chain of proper subspaces of $V$ is finite (since $\dim_\F V <
\infty$), whence its upper bound is in the chain (note that this fails if
$|I| = \dim_\F V = \infty$). So for every chain of subspaces, remove all
of them except the upper bound.

We are left with $\{ V_i : i \in I' \}$, where if $i \neq j$ in $I'$,
then $V_j \nsubseteq V_i$, or $V_i \cap V_j \subsetneq V_j$. Now use the
induction hypothesis: no $V_j$ is a union of ``$I$-many" (hence
``$I'$-many") proper subspaces. So
\[ V_j \supsetneq \bigcup_{i \in I', i \neq j} (V_j \cap V_i) = V_j \bigcap
\bigcup_{i \in I', i \neq j} V_i, \]

\noindent whence no $V_j$ is contained in the union of the others, as
desired.\medskip

\item Having found such a subcollection, we now obtain the desired
contradiction:

For all $i \in I'$, choose $v_i \in V_i$ such that $v_i \notin V_j\
\forall i \neq j$. There are at least two such, so choose $v_1 = v_{i_1},
v_2 = v_{i_2}$, with $i_1 \neq i_2$ in $I'$. Now consider $S := \{ v_1 +
\alpha v_2 : \alpha \in \F \} \coprod \{ v_2 \}$. Since $V = \bigcup_{i
\in I'} V_i$, for each vector $v \in S$, choose some $i$ such that $v \in
V_i$. This defines a function $f : \F \coprod \{ \infty \} \to I'$, and
this is not injective by assumption. Thus some two elements of $S$ are in
the same $V_i$, and we can solve this system of linear equations to infer
that both $v_1$ and $v_2$ are in $V_i$. Hence $i_1 = i = i_2$, a
contradiction.
\end{enumerate}\medskip

\item We now show the result for general $k$. We have two cases. If $\F$
is infinite, then we are done by the previous part and the final part of
Remark \ref{R2}.
The other case is when $\F$ is finite - say $\F = \F_q$ - whence $I$ is
finite. In this case, take any set of subspaces $V_1, \dots, V_i$ of
codimension $\geq k$, with $i = |I| < q^k + 1$; we are to show that
$\bigcup_j V_j \subsetneq V$.

We now reduce the situation to that of a finite-dimensional quotient $V'$
of $V$ as follows. First, we may increase each $V_i$ to a codimension $k$
subspace. Next,
\begin{equation}\label{E2}
\dim_\F \left( V_1 / (V_1 \cap V_2) \right) = \dim_\F \left( (V_1 +
V_2) / V_2 \right) \leq \dim_\F (V / V_2) < \infty,
\end{equation}

\noindent whence $\dim_\F V / (V_1 \cap V_2) \leq \dim_\F (V / V_1) +
\dim_\F (V / V_2) < \infty$. Now proceed inductively to show that $V_0 :=
\bigcap_{j=1}^i V_j$ has finite codimension in $V$; more precisely,
$\dim_\F (V / V_0)$ is bounded above by $\sum_{j=1}^i \dim_\F (V / V_j)$.

Thus, we quotient by $V_0$, and end up with codimension-$k$ subspaces
$V'_j$ covering the finite-dimensional quotient $V' = V / V_0$. Now if
$\dim_\F V' = n$, then we are covering $q^n - 1$ nonzero vectors in $V'$
by proper subspaces $V'_j$, each with at most $q^{n-k} - 1$ nonzero
vectors. Thus the number of subspaces needed, is at least $\geq \frac{q^n
- 1}{q^{n-k} - 1} > q^k$, as claimed.
\end{enumerate}
\end{proof}

The following result concludes the proof for infinite fields, by the last
part of Remark \ref{R2}.

\begin{lemma}
If $\F$ and $\dim_\F V$ are both infinite, then $V$ is a countable union
of proper subspaces.
\end{lemma}

\begin{proof}
(As for Proposition \ref{Palg}.) Fix any (infinite) basis $B$ of $V$, and
a sequence of proper subsets
$\emptyset = B_0 \subset B_1 \subset \dots$ of $B$, whose union is $B$.
Now define $V_n$ to be the span of $B_n$ for all $n$. Then the $V_n$'s
provide a cover of $V$ by proper subspaces, each of infinite codimension
in $V$.
\end{proof}
%}}}

%{{{1 Section 3 - Proof for finite fields
\section{Proof for finite fields}

We now complete the proof. In what follows, we will crucially use some
well-known results on partitions of finite vector spaces. These are found
in \cite[Lemmas 2,4]{Bu}, though the first part below was known even
before \cite{Beu}).

\begin{lemma}\label{L1}
Suppose $V$ is an $n$-dimensional vector space over the finite field $\F
= \F_q$ (for some $q,n \in \N$), and we also fix $d \in \N$.
\begin{enumerate}
\item $V$ can be partitioned using only $d$-dimensional subspaces, if and
only if $d | n$. (The number of such subspaces is $(q^n - 1) / (q^d -
1)$.)

\item Let $1 < d < n/2$. Then $V$ can be partitioned into one
$(n-d)$-dimensional subspace, and $q^{n-d}$ subspaces of dimension $d$.
\end{enumerate}
\end{lemma}

We now show most of the main result, for finite fields. 

\begin{prop}\label{Pq}
Suppose $V$ is a finite set. Then $V$ is covered by ``$I$-many" subspaces
of codimension at least $k$, if and only if $I \geq \proj(V) / (\proj(V /
\F^k))$.
\end{prop}

\begin{proof}
If $V$ is finite, then so are $\F$ and $\dim_\F V$. We may also assume
that the subspaces that cover $V$ are of codimension exactly equal to
$k$. Now suppose $V$ is covered by ``$I$-many" such subspaces, and
$\dim_\F V = n \in \N$. Then we need to cover $q^n-1$ nonzero vectors by
proper subpaces, each with $q^{n-k}-1$ nonzero vectors, whence
\[ I \geq \frac{q^n - 1}{q^{n-k} - 1} = \frac{\proj(V)}{\proj(V / \F^k)},
\]

\noindent as required.

We now show the converse: if $\dim_\F V = n$, and $(n-k) | n$, then we
are done by the first part of Lemma \ref{L1}, since there exists a
partition.
In the other case, we illustrate the proof via an example that can easily
be made rigorous. We first fix $\F = \F_q$; now suppose $n=41$ and
$k=29$. We must, then, find $\lceil (q^{41} - 1)/(q^{12} - 1) \rceil =
q^{29} + q^{17} + q^5 + 1$ subspaces of codimension 29, that cover
$\F^{41}$.

Now set $d=12$ and apply the second part of Lemma \ref{L1}; thus,
\[ \F^{41} = \F^{29} \coprod (\F^{12})^{\coprod q^{29}}. \]

\noindent In other words, we have $q^{29}$ 12-dimensional subspaces, and
one extra subspace of dimension 29. Now apply the same result again (with
$d = 12$ and replacing $n=41$ by 29) to get
\[ \F^{41} = \F^{17} \coprod (\F^{12})^{\coprod q^{17}} \coprod
(\F^{12})^{\coprod q^{29}}. \]

\noindent (For a general $n,k$, apply the result repeatedly with $d =
n-k$ and $n$ replaced by $n-d, n-2d, \dots$, until there remains one
subspace of dimension between $d$ and $2d$, and ``almost disjoint"
subspaces of codimension $k$.)\medskip

To conclude the proof, it suffices to cover $V_1 = \F_q^{17}$ with $q^5 +
1$ subspaces of dimension 12. To do this, fix some 7-dimensional subspace
$V_0$ of $V_1$, and consider $V_1 / V_0 \cong \F_q^{10}$. By the first
part of Lemma \ref{L1}, this has a partition into $(q^5 + 1)$
5-dimensional subspaces. Lift this partition to $V_1$; this provides the
desired (remaining) $q^5 + 1$ subspaces of codimension $29$ in $\F^{41}$.
\end{proof}

The last part of the main result can now be shown, using this result.

\begin{proof}[Proof of Theorem \ref{Tvecunion}]
The above results show the theorem except in the case when $\F$ is
finite, but $\dim_\F V$ is not. In this case, by Proposition \ref{P1}, we
only need to show that $V$ can be covered by $q^k + 1$ subspaces of
codimension $k$. To see this, quotient $V$ by a codimension $2k$
subspace $V_0$; now by Proposition \ref{Pq}, $V / V_0$ can be covered by
\[ \frac{\proj(V / V_0)}{\proj((V / V_0) / \F^k)} = \frac{(q^{2k} - 1) /
(q-1)}{(q^k - 1) / (q-1)} = q^k + 1 \]

\noindent subspaces of codimension $k$. Lift these to $V$ for the desired
cover.
\end{proof}
%}}}

\subsection*{Acknowledgments}
I thank Sunil Chebolu, Siddharth Joshi, Anindya Sen, and Rajeev Walia for
valuable discussions, and the anonymous referee for suggestions, which
helped in improving the exposition.

%{{{1 Bibliography
\medskip
%}}}

\begin{thebibliography}{ba}
\bibitem{Beu}
A.~Beutelspacher, {\em Partitions of finite vector spaces: An application
  of the Frobenius number in geometry}, Arch.~Math.~\textbf{31} (1978),
  202--208.

\bibitem{Bu}
T.~Bu, {\em Partitions of a vector space}, Discrete Math.~\textbf{31}
  (1978), 79--83.

\bibitem{ZSSSE}
S.I.~El-Zanati, G.F.~Seelinger, P.A.~Sissokho, L.E.~Spence, and C.~Vanden
  Eynden, {\em Partitions of finite vector spaces into subspaces},
  Journal of Combinatorial Designs \textbf{16} (2008), 329--341.

\bibitem{Hed1}
O.~Heden, {\em The Frobenius number and partitions of a finite vector
  space},\hfill\break Arch.~Math.~\textbf{42} (1984), 185--192.

\bibitem{Hed2}
\bysame, {\em On partitions of finite vector spaces of small dimensions},
  Arch.~Math.~\textbf{43} (1984), 507--509.

\bibitem{Khchinese}
A.~Khare, {\em Abelian groups as unions of proper subgroups}, preprint, 
  \href{http://arxiv.org/abs/0906.1023}{\tt arxiv.org/abs/0906.1023}.

\bibitem{Luh}
J.~Luh, {\em On the representation of vector spaces as a finite union of
  subspaces}, Acta Math.~Acad.~Sci.~Hungar.~\textbf{23} (1972), 341--342.

\bibitem{Ani}
A.~Sen, personal communication, January 2008.
\end{thebibliography}
\end{document}